\documentclass[12pt,a4paper,reqno]{amsart}

\usepackage[lmargin=1.4 in, rmargin= 1.4in, tmargin=1.7in]{geometry}

\title{The category of necklaces is a test category}

\author{Arne Mertens}
\address[Arne Mertens]{Universiteit Antwerpen, Departement Wiskunde, Middelheimcampus,
Middelheimlaan 1, 2020 Antwerp, Belgium}
\email{arne.mertens@uantwerpen.be}

\usepackage{graphicx}
\usepackage[english]{babel}
\usepackage{amsmath}
\usepackage{amssymb}
\usepackage{amsthm}
\usepackage[all,cmtip]{xy}
\usepackage{cancel}
\usepackage[alpine]{ifsym}
\usepackage{enumerate}
\usepackage{stmaryrd}
\usepackage{tikz}
\usepackage{tikz-cd}
\usepackage{quiver}
\usetikzlibrary{matrix}
\usepackage[toc,page]{appendix}
\usepackage{hyperref}
\usepackage{mathtools}
\usepackage{float}
\usepackage{hhline}
\usepackage[style=alphabetic]{biblatex}
\DeclareMathOperator{\id}{id}

\DeclareMathOperator{\Fun}{Fun}

\DeclareMathOperator*{\colim}{colim}

\DeclareMathOperator{\Set}{Set}

\DeclareMathOperator{\SSet}{SSet}

\DeclareMathOperator{\Cat}{Cat}

\newcommand{\nec}{\mathcal{N}ec} 

\newtheorem{Thm}{Theorem}[section]
\newtheorem*{Thm*}{Theorem}

\newtheorem{Prop}[Thm]{Proposition}

\theoremstyle{definition}
\newtheorem{Def}[Thm]{Definition}

\theoremstyle{remark}

\begin{document}

\begin{abstract}
In this short note, we prove that the category of necklaces is a test category. Hence presheaves on necklaces model homotopy types. The proof is analogous to that for the category of cubes with connections.
\end{abstract}

\maketitle


\section{Introduction}

In this short note, we prove that the category $\nec$ of necklaces of \cite{dugger2011rigidification} is a test category in the sense of \cite{cisinski2006prefaisceaux}. We first recall some preliminaries on test categories and necklaces, and then prove our main result (Theorem \ref{theorem: nec is test category}).

\section{Preliminaries}

\subsection{Test categories}

We follow the terminology and conventions \cite{jardine2006categorical}. Be aware that these conventions sometimes differ from those of the original reference \cite{cisinski2006prefaisceaux}.

Let $\mathcal{A}$ be a fixed small category. We refer to presheaves $\mathcal{A}^{op}\to \Set$ as \emph{$\mathcal{A}$-sets}. Given $a\in \mathcal{A}$, let us denote the representable $\mathcal{A}$-set by $\hat{a} = \mathcal{A}(-,a)$.

Consider the functor
\begin{equation}\label{diagram: slice functor}
\iota: \mathcal{A}\rightarrow \Cat: a\mapsto \mathcal{A}/a
\end{equation}
Left Kan extension of $\iota$ along the Yoneda embedding $\mathcal{A}\hookrightarrow \Set^{\mathcal{A}^{op}}$ induces an adjunction
$$
\iota_{*} = \int_{\mathcal{A}}: \Set^{\mathcal{A}^{op}}\leftrightarrows \Cat: \iota^{*}
$$
For every $\mathcal{A}$-set $X$, $\int_{\mathcal{A}}X$ is precisely the category of elements of $X$. Given a small category $\mathcal{C}$, we have $\iota^{*}(\mathcal{C}) = \Fun(\mathcal{A}/-,\mathcal{C}): \mathcal{A}^{op}\to \Set$.

\begin{Def}
Let $N: \Cat\hookrightarrow \SSet$ denote the nerve functor.

A functor $F: \mathcal{C}\rightarrow \mathcal{D}$ between small categories is \emph{aspherical} or \emph{homotopy cofinal} if for all $d\in \mathcal{D}$, the simplicial set $N(F/d)$ is weakly contractible.  A small category $\mathcal{C}$ is \emph{aspherical} if the functor $\mathcal{C}\rightarrow *$ is aspherical, that is if $N(\mathcal{C})$ is weakly contractible.

An $\mathcal{A}$-set is called \emph{aspherical} if the forgetful functor $\int_{\mathcal{A}}X\to \mathcal{A}$ is aspherical. This is equivalent to the category $\int_{\mathcal{A}}(\hat{a}\times X)$ being aspherical for all $a\in \mathcal{A}$.
\end{Def}

\begin{Def}[Lemmas 2.3 and 2.4, \cite{jardine2006categorical}]
$\mathcal{A}$ is called a
\begin{enumerate}[1.]
\item \emph{local test category} if for every small category $\mathcal{D}$ with a terminal object, the $\mathcal{A}$-set $\iota^{*}(\mathcal{D})$ is aspherical.
\item \emph{test category} if it is a local test category and aspherical.
\end{enumerate}
\end{Def}

\subsection{Necklaces}

We follow the terminology and conventions of \cite[\S 3]{dugger2011rigidification} as well as \cite[\S 3.2]{mertens2026nerves}.

Let $\SSet_{*,*} = \SSet_{\partial\Delta^{1}/}$ denote the category of bipointed simplicial sets. Given two bipointed simplicial sets $K_{a,b}$ and $L_{c,d}$, we define their \emph{wedge sum} by the following coequalizer, which we consider to be bipointed at $a$ and $d$:
$$
\Delta^{0}\overset{b}{\underset{c}{\rightrightarrows}} K_{a,b}\amalg L_{c,d}\twoheadrightarrow K\vee L
$$

We always consider the standard simplices $\Delta^{n}$ to be bipointed at $0$ and $n$.

\begin{Def}
The category $\nec$ of \emph{necklaces} is the full subcategory of $\SSet_{*,*}$ spanned by all bipointed simplicial sets of the form
$$
T = \Delta^{n_{1}}\vee \dots\vee \Delta^{n_{k}}
$$
for $k\geq 0$ and each $n_{i} > 0$. The endpoints of each of the $\Delta^{n_{i}}$ in $T$ are called the \emph{joints} of $T$. We define the \emph{dimension} of $T$ as the number of non-joint vertices of $T$, and denote it $\dim(T)$.
\end{Def}

Under the wedge sum $\vee$, the category $\nec$ is monoidal with its unit given by $\Delta^{0}$. As a monoidal category, $\nec$ is generated by the following three classes of maps:
\begin{itemize}
\item the coface map $\delta_{j}: \Delta^{n-1}\hookrightarrow \Delta^{n}$ for $0 < j < n$,
\item the codegeneracy map $\sigma_{i}: \Delta^{n+1}\to \Delta^{n}$ for $0\leq i\leq n$, and
\item the inclusion $\nu_{k,l}: \Delta^{k}\vee \Delta^{l}\hookrightarrow \Delta^{k+l}$ which is bijective on vertices.
\end{itemize}

There is a strong monoidal functor
\begin{equation}\label{diagram: dimension functor}
\dim: \nec\to \Cat: T\mapsto [1]^{\dim(T)}
\end{equation}
where $[1] = \{0 < 1\}$ is the walking arrow category. The functor $\dim$ is defined on generating maps as follows:
\begin{itemize}
\item $\dim(\delta_{j}) = \delta^{\square}_{j,0}$
\item $\dim(\nu_{k,l}) = \delta^{\square}_{k,1}$
\item $\dim(\sigma_{i}) =
\begin{cases}
\sigma^{\square}_{1} & \text{if }i = 0\\
\gamma^{\square}_{i} & \text{if }0 < i < n\\
\sigma^{\square}_{n} & \text{if }i = n
\end{cases}$
if $n > 0$
\item $\dim(\sigma_{0}) = \id_{[0]}$ for $\sigma_{0}: \Delta^{1}\to \Delta^{0}$
\end{itemize}
where $\delta^{\square}_{i,\epsilon}$, $\sigma^{\square}_{i}$ and $\gamma^{\square}_{i}$ are maps of cubes, which are given by
\begin{align*}
\delta^{\square}_{i,\epsilon}&: [1]^{n-1}\rightarrow [1]^{n}: (\epsilon_{1},\dots,\epsilon_{n-1})\mapsto (\epsilon_{1},\dots,\epsilon_{i-1},\epsilon,\epsilon_{i},\dots,\epsilon_{n}) &&\text{for }1\leq i\leq n, \epsilon\in \{0,1\}\\
\sigma^{\square}_{i}&: [1]^{n+1}\rightarrow [1]^{n}: (\epsilon_{1},\dots,\epsilon_{n+1})\mapsto (\epsilon_{1},\dots,\epsilon_{i-1},\epsilon_{i+1},\dots,\epsilon_{n+1}) &&\text{for }1\leq i\leq n+1\\
\gamma^{\square}_{i}&: [1]^{n+1}\rightarrow [1]^{n}: (\epsilon_{1},\dots,\epsilon_{n+1})\mapsto (\epsilon_{1},\dots,\max(\epsilon_{i},\epsilon_{i+1}),\dots,\epsilon_{n}) &&\text{for }1\leq i\leq n-1
\end{align*}

\section{Proof}

We make use of the following general `yoga' as described in \cite{jardine2006categorical}.

\begin{Prop}[Remark 3.12, \cite{jardine2006categorical}]\label{proposition: amenable functor}
Let $i: \mathcal{A}\to \Cat$ be a functor and assume that
\begin{enumerate}[(a)]
\item for every $a\in \mathcal{A}$, $i(a)$ has a terminal object,
\item for every small category $\mathcal{D}$ with a terminal object, the $\mathcal{A}$-set $i^{*}(\mathcal{D})$ is aspherical, and
\item the category $\mathcal{A}$ is aspherical. 
\end{enumerate}
Then $\mathcal{A}$ is a test category.
\end{Prop}

This allows to essentially replace the functor $\iota: \mathcal{A}\to \Cat$ \eqref{diagram: slice functor} by another which may be easier to describe. In the case of necklaces, we can consider the dimension functor $\dim: \nec\to \Cat$ \eqref{diagram: dimension functor}. It induces an adjunction
$$
\dim_{*}: \Set^{\nec^{op}}\leftrightarrows \Cat: \dim^{*}
$$
where $\dim_{*}(X) = \colim_{T\in \nec}^{X_{T}}[1]^{\dim(T)}$ and $\dim^{*}(\mathcal{C}) = \Fun([1]^{\dim(-)},\mathcal{C})$.

%

The next proof follows essentially the same strategy as for the category of cubes with connections, presented in \cite[Lemmas 3.10 and 3.11]{jardine2006categorical}.

\begin{Thm}\label{theorem: nec is test category}
$\nec$ is a test category.
\end{Thm}
\begin{proof}
We prove the statement by verifying conditions $(a)$-$(c)$ of Proposition \ref{proposition: amenable functor} for the functor $\dim: \nec\to \Cat$. First note that for every $T\in \nec$, the category $[1]^{\dim(T)}$ indeed has a terminal object. The category $\nec$ also has a terminal object $\Delta^{0}$, whereby it is aspherical. Hence $(a)$ and $(c)$ are satisfied.

To show $(b)$, let $\mathcal{D}$ be a small category with a terminal object $t$. We must show that for any necklace $T$, the category $\mathcal{D}_{T} = \int_{\nec}(\dim^{*}(\mathcal{D})\times \hat{T})$ is aspherical. Take an arbitrary object of $\mathcal{D}_{T}$, given by a pair $(F,f)$ with $f: U\to T$ a necklace map and $F: [1]^{\dim(U)}\rightarrow \mathcal{D}$ a functor. Further let $d = \dim(U)$ and define $f^{*} = f(\sigma_{0}\sigma_{0}\vee \id_{U}): \Delta^{2}\vee U\to U\to T$ and
$$
F^{*}: [1]^{d+1}\to \mathcal{D}: (\epsilon_{0},\dots,\epsilon_{d+1})\mapsto 
\begin{cases}
F(\epsilon_{0},\dots,\epsilon_{d}) & \text{if }\epsilon_{d+1} = 0\\
t & \text{if }\epsilon_{d+1} = 1
\end{cases}
$$
Then we have the following commutative diagrams in $\Cat$ and $\nec$ respectively:
\[\begin{tikzcd}
	{[1]^{d}} & {[1]^{d}} & {[1]^{d+1}} & {[1]^{d}} & {[1]^{d}} \\
	& & {\mathcal{D}}
	\arrow["{=}"', from=1-2, to=1-1]
	\arrow["{F}"', from=1-1, to=2-3]
	\arrow["{\delta^{\square}_{1,0}}", from=1-2, to=1-3]
	\arrow["{=}", from=1-4, to=1-5]
	\arrow["{\delta^{\square}_{1,1}}"', from=1-4, to=1-3]
	\arrow["F"{description}, from=1-2, to=2-3]
	\arrow["{F^{*}}"{description}, from=1-3, to=2-3]
	\arrow["t"{description}, from=1-4, to=2-3]
	\arrow["t", from=1-5, to=2-3]
\end{tikzcd}\]
\[\begin{tikzcd}
	U & \Delta^{1}\vee U & {\Delta^{2}\vee U} & {\Delta^{1}\vee \Delta^{1}\vee U} & U \\
	& & {T}
	\arrow["\sigma_{0}\vee U"', from=1-2, to=1-1]
	\arrow["f"', from=1-1, to=2-3]
	\arrow["{\delta_{1}\vee U}", from=1-2, to=1-3]
	\arrow["{\sigma_{0}\vee \sigma_{0}\vee U}", from=1-4, to=1-5]
	\arrow["{\nu_{1,1}\vee U}"', from=1-4, to=1-3]
	\arrow["f", from=1-5, to=2-3]
	\arrow[from=1-4, to=2-3]
	\arrow[from=1-2, to=2-3]
	\arrow["f^{*}"{description}, from=1-3, to=2-3]
\end{tikzcd}\]
Then we have induced endofunctors on $\mathcal{D}_{T}$:
\begin{align*}
& H: (F,f)\mapsto (F,f(\sigma_{0}\vee \id_{U})) && I: (F,f)\mapsto (F^{*},f^{*})\\
& J: (F,f)\mapsto (t,f(\sigma_{0}\vee \sigma_{0}\vee \id_{U}) && K: (F,f)\mapsto (t,f)
\end{align*}
along with natural transformations:
$$
\id_{\mathcal{D}_{T}}\leftarrow H\rightarrow I\leftarrow J\rightarrow K
$$
which induce homotopies between simplicial maps after applying the nerve. Hence, it follows that $N(K)$ is a weak homotopy equivalence. Now note that $G$ can be factored as $\mathcal{D}_{T}\xrightarrow{\alpha} \nec/T\xrightarrow{\beta} \mathcal{D}_{T}$ where $\alpha: (F,f)\mapsto f$ and $\beta: f\mapsto (t,f)$. Since clearly $\alpha\beta = \id_{\nec/T}$, it follows from the 2-out-of-6-property that $\alpha$ and $\beta$ are weak homotopy equivalence between $\mathcal{D}_{T}$ and $\nec/T$. Now $\nec/T$ is clearly aspherical and thus so is $\mathcal{D}_{T}$.
\end{proof}

\printbibliography

@article {cisinski2006prefaisceaux,
    AUTHOR = {Cisinski, Denis-Charles},
     TITLE = {Les pr\'{e}faisceaux comme mod\`eles des types d'homotopie},
   JOURNAL = {Ast\'{e}risque},
  FJOURNAL = {Ast\'{e}risque},
    NUMBER = {308},
      YEAR = {2006},
     PAGES = {xxiv+390},
      ISSN = {0303-1179,2492-5926},
      ISBN = {978-2-85629-225-9},
   MRCLASS = {55-02 (18F20 18G50 55P60 55U35)},
  MRNUMBER = {2294028},
MRREVIEWER = {Philippe\ Gaucher},
}

@article {dugger2011rigidification,
    AUTHOR = {Dugger, Daniel and Spivak, David I.},
     TITLE = {Rigidification of quasi-categories},
   JOURNAL = {Algebr. Geom. Topol.},
  FJOURNAL = {Algebraic \& Geometric Topology},
    VOLUME = {11},
      YEAR = {2011},
    NUMBER = {1},
     PAGES = {225--261},
   MRCLASS = {55U40 (18G30)},
  MRNUMBER = {2764042},
MRREVIEWER = {G\'{e}rald Gaudens},
       %DOI = {10.2140/agt.2011.11.225},
}

@article {jardine2006categorical,
    AUTHOR = {Jardine, J. F.},
     TITLE = {Categorical homotopy theory},
   JOURNAL = {Homology Homotopy Appl.},
  FJOURNAL = {Homology, Homotopy and Applications},
    VOLUME = {8},
      YEAR = {2006},
    NUMBER = {1},
     PAGES = {71--144},
      ISSN = {1532-0073,1532-0081},
   MRCLASS = {55P60 (14F35 18F20 55U35)},
  MRNUMBER = {2205215},
MRREVIEWER = {Richard\ John\ Steiner},
       %URL = {http://projecteuclid.org/euclid.hha/1140012467},
}

@article{mertens2026nerves,
author = {Mertens, Arne},
title = {Nerves of enriched categories via necklaces},
journal = {Journal of the London Mathematical Society},
volume = {113},
number = {6},
pages = {e70573},
%doi = {https://doi.org/10.1112/jlms.70573},
%url = {https://londmathsoc.onlinelibrary.wiley.com/doi/abs/10.1112/jlms.70573},
eprint = {https://londmathsoc.onlinelibrary.wiley.com/doi/pdf/10.1112/jlms.70573},
year = {2026}
}

\end{document}